\documentclass[12pt,leqno,draft]{article}
\usepackage{bbm}
\usepackage{amsfonts}
\usepackage{stmaryrd}
\usepackage{amsmath, amsthm, amsfonts, amssymb, color}
\usepackage{mathrsfs}
\usepackage{color}
\newtheorem{thm}{Theorem}[section]

\newtheorem{lem}[thm]{Lemma}
\newtheorem{prp}[thm]{Proposition}

\theoremstyle{definition}

\newtheorem{rem}[thm]{Remark}

\newcommand{\scr}[1]{\mathscr #1}
\definecolor{wco}{rgb}{0.5,0.2,0.3}

\numberwithin{equation}{section} \theoremstyle{remark}

\newcommand{\ua}{\uparrow}

\title{{\bf Existence of invariant probability measures for functional McKean-Vlasov SDEs} \footnote{Supported in
 part by  NNSFC (11771326, 11831014, 12071340).} }
\author{
{\bf  Jianhai Bao\footnote{Center for Applied Mathematics, Tianjin
University, Tianjin 300072, China; jianhaibao@tju.edu.cn}, Michael Scheutzow\footnote{Institut f\"ur Mathematik, MA 7-5, Fakult\"at II, Technische Universit\"at Berlin,  Stra\ss e  des 17. Juni 136, 10623  Berlin, Germany; ms@math.tu-berlin.de }, Chenggui Yuan\footnote{Department of Mathematics, Swansea University, Bay Campus SA1 8EN, UK; C.Yuan@swansea.ac.uk} }\\
}
\date{}
\begin{document}
\allowdisplaybreaks
\def\R{\mathbb R}  \def\ff{\frac}  \def\B{\mathbf
B}
\def\N{\mathbb N} \def\kk{\kappa} \def\m{{\bf m}}
\def\ee{\varepsilon}\def\ddd{D^*}
\def\dd{\delta} \def\DD{\Delta} \def\vv{\varepsilon} \def\rr{\rho}
\def\<{\langle} \def\>{\rangle} \def\GG{\Gamma} \def\gg{\gamma}
  \def\nn{\nabla} \def\pp{\partial} \def\E{\mathbb E}
\def\d{\text{\rm{d}}} \def\bb{\beta} \def\aa{\alpha} \def\D{\scr D}
  \def\si{\sigma} \def\ess{\text{\rm{ess}}}
\def\beg{\begin} \def\beq{\begin{equation}}  \def\F{\scr F}
\def\Ric{\text{\rm{Ric}}} \def\Hess{\text{\rm{Hess}}}
\def\e{\text{\rm{e}}} \def\ua{\underline a} \def\OO{\Omega}  \def\oo{\omega}
 \def\tt{\tilde} \def\Ric{\text{\rm{Ric}}}
\def\cut{\text{\rm{cut}}} \def\P{\mathbb P} \def\ifn{I_n(f^{\bigotimes n})}
\def\C{\scr C}   \def\G{\scr G}   \def\aaa{\mathbf{r}}     \def\r{r} \def\CC{\mathcal{C}}
\def\gap{\text{\rm{gap}}} \def\prr{\pi_{{\bf m},\varrho}}  \def\r{\mathbf r}
\def\Z{\mathbb Z} \def\vrr{\varrho} \def\ll{\lambda}
\def\L{\scr L}\def\Tt{\tt} \def\TT{\tt}\def\II{\mathbb I}
\def\i{{\rm in}}\def\Sect{{\rm Sect}}  \def\H{\mathbb H}
\def\M{\scr M}\def\Q{\mathbb Q} \def\texto{\text{o}} \def\LL{\Lambda}
\def\Rank{{\rm Rank}} \def\B{\scr B} \def\i{{\rm i}} \def\HR{\hat{\R}^d}
\def\to{\rightarrow}\def\l{\ell}\def\iint{\int}
\def\EE{\scr E}\def\no{\nonumber}
\def\A{\scr A}\def\V{\mathbb V}\def\osc{{\rm osc}}
\def\BB{\mathbb B}\def\Ent{{\rm Ent}}\def\W{\mathbb W}
\def\U{\scr U}\def\8{\infty} \def\si{\sigma}%\def\1{\lesssim}
\def\33{\interleave}
\def\M{\mathcal{M}}
\def\K{\mathcal{K}}
\newcommand{\1}{\mathbbm{1}}
\renewcommand{\bar}{\overline}
\renewcommand{\tilde}{\widetilde}
\maketitle
\begin{abstract}
  We show existence of an invariant probability measure for a class of functional McKean-Vlasov SDEs by applying Kakutani's fixed point theorem to a suitable class of probability measures on a space of continuous functions.
  Unlike some previous works \cite{AD,Wangb},  we do not assume a monotonicity condition to hold.
  Further, our conditions are even weaker than some results in the literature on invariant probability measures for functional SDEs %stochastic functional differential equations
  without dependence on the law of the solution \cite{ESV}.

\end{abstract} \noindent

 \section{Introduction and main result}
 Many classical functional SDEs generate  Markov processes on a space of functions and it is natural to ask for the existence and uniqueness of invariant probability measures (IPMs) of these processes. A typical way to establish existence of an IPM of an SDE with or without delay is to show boundedness in probability of the solutions and then to apply the Krylov-Bogoliubov theorem (see e.g. \cite[Theorem 3.1.1]{DZ96}). Boundedness in probability is usually guaranteed by constructing a suitable
 {\em Lyapunov} function (sometimes called  Veretennikov-Khasminskii condition); see, for example, \cite{But,BS,HMS,Sch}. Some authors establish the existence of IPMs without using the Krylov-Bogoliubov theorem: \cite{BYY} uses the remote start method  and \cite{Sch} uses non-degeneracy and recurrence properties of the Markov process.

%The Krylov-Bogoliubov theorem (see e.g. \cite[Theorem 3.1.1]{DZ96}) is a popular and powerful tool to establish existence of an IPM. It can be applied to Feller processes provided one can
% establish some appropriate stochastic boundedness of the marginals of the processes, for example, by constructing a suitable Lyapunov function; see, for instance, \cite{ESV,KW}. \red{Next, if the transition kernel of the segment process associated with a functional SDE  is a Cauchy sequence in the Wasserstein space, then the corresponding limit distribution is indeed an IPM; see, for example, \cite{But,BS}. Furthermore, under the so-called Veretennikov-Khasminskii condition, \cite[Theorem 4]{Sch} showed that the segment process of a nondegenerate functional SDE is $\varphi$-recurrent which guarantees existence of   an IPM.}

 If the transition kernels starting from different initial distributions are contractive under a Wasserstein metric or the total variation metric, then the classical Banach fixed point theorem yields  that the associated transition semigroup has a unique IPM; see, for example, \cite{HMS}. Such conditions are however too strong if one is only interested in the existence of an IPM.
 %The remote start method was  also employed  to study existence of IPMs; see e.g. \cite{BYY}.
 %Concerning the approaches mentioned above, we emphasize that the {\it Markov property} of the corresponding segment processes, which also implies the associated {\it Markov semigroup  is linear},  plays a crucial role in investigating existence of IPMs.

 Over the past few decades, McKean-Vlasov SDEs (with or without delay) have gained considerable attention. They are applied in the social sciences, economics, neuroscience, engineering, and finance; see, for example, \cite{MSSZ}
 or the monograph \cite{CD}. They  often appear as limiting equations for $N$ weakly interacting particles when $N \to \infty$. A McKean-Vlasov SDE is an SDE whose coefficients depend not only on the state of the process
 but also on its distribution. It is of the form %More precisely, a McKean-Vlasov SDE on $\R^d$ solves the following SDE
 \begin{equation}\label{E0}
 \d X(t)=b(X(t),\L_{X(t)})\d t+\si(X(t),\L_{X(t)})\d W(t),
\end{equation}
 where $\L_{X(t)}$ denotes the distribution of $X(t)$ and $(W(t))$ is an $m$-dimensional Brownian motion.
 In contrast to classical SDEs, the distribution of the McKean-Vlasov SDE \eqref{E0} solves a {\em nonlinear} Fokker-Planck equation and     the solution process of \eqref{E0}  is a {\it nonlinear Markov process in the sense of McKean} (but not a classical linear Markov process). Further, there is no associated (linear) semigroup but the distribution flow $P_t^*\mu:=\L_{X_t^\mu}$ generates a nonlinear semigroup. Therefore, the approaches taken in e.g. \cite{BYY,ESV,HMS,KW} to address existence of IPMs for linear Markov processes (or semigroups) cannot be applied to  the McKean-Vlasov SDE \eqref{E0}.

 %On the other hand,   the Markov operator generated by the solution process of \eqref{E0} is not any more a semigroup as shown in \cite{Wangb} although the distribution flow $P_t^*\mu:=\L_{X_t^\mu}$ is a nonlinear semigroup, where $\L_{X_t^\mu}$ is the distribution of $X_t$ with the initial distribution $\L_{X_0}=\mu$. Based on the analysis above, the approaches applied in e.g. \cite{BYY,ESV,HMS,KW} to address existence of IPMs for linear Markov processes (or semigroups) are invalid   for the McKean-Vlasov SDE \eqref{E0}.

 Let us briefly review some previous  approaches to show existence of IPMs for McKean-Vlasov SDEs.
 By applying the generalized Banach fixed-point theorem (see e.g. \cite[Theorem 9.A, p449]{Zeidler}), \cite[Theorem 3]{AD} investigated existence of IPMs for
 semi-linear McKean-Vlasov stochastic partial differential equations (SPDEs for short). By employing the shift coupling approach and using contractivity under
 Wasserstein distance, \cite[Theorem 3.1]{Wangb} discussed existence of IPMs for the McKean-Vlasov SDE \eqref{E0}. For further extensions to functional
 McKean-Vlasov SPDEs, we refer to \cite[Theorem 2.3]{RW}.  %\red{Furthermore, via the Krylov-Bigiliubov Theorem,
 The existence of IMPs for semigroups on $C_b(P_2(\R^d))$ was investigated in \cite[Theorem 4.1]{HSS} for McKean-Vlasov SDEs under measure dependent Lyapunov
 conditions.

 Some authors assume a monotonicity condition to show existence and uniqueness of an IPM, e.g.\cite{AD,Wangb}. Such a condition is, however, very strong if one
 is only interested in the  existence of an IPM.
 % As we know, with regard to the classical SDEs, the monotone condition can guarantee not only existence but also uniqueness of IPMs. Therefore, compared with the classical SDEs, the sufficient condition imposed in \cite{AD,Wangb} to ensure existence of IPMs is rather strong. Moreover, via a perturbation method, \cite{Bu} studied existence of IPMs for McKean-Vlasov SDEs with additive noises, where the drift term is Lipschitz continuous with respect the space variable and also Lipschitz under {\it the total variation metric} with respect to the measure variable.
 Motivated by the previous works,   in this paper, we shall focus on a functional McKean-Vlasov SDE (see \eqref{E1} below), which is much more general than the McKean-Vlasov SDE \eqref{E0},   and investigate existence of IPMs under much weaker conditions than those imposed in \cite{AD,Bu,Wangb}. In particular, we will not assume a monotonicity condition.

 Let us point out that we will not address the question of uniqueness of IPMs which clearly requires stronger assumptions than those in our main result.
 %Our conditions are clearly too weak to guarantee uniqueness

Before we present our set-up and results, let us introduce  some notation.  For a subinterval $U\subset \R$,   $C(U;\R^d)$  stands for the collection of all continuous functions $f:U\to \R^d$. For a fixed  finite number $r_0>0,$ set $\C:=C([-r_0,0];\R^d)$, which is endowed with the uniform norm $$\|f\|_\8:=\sup_{-r_0\le\theta\le0 }|f(\theta)|,~~~f\in\C.$$
   For  $f\in C([-r_0,\8);\R^d)$ and $t\ge0$, let $f_t\in\C$ be defined by $f_t(\theta)=f(t+\theta), \theta\in[-r_0,0]$. Often, $(f_t)_{t\ge0}$ is called the {\em segment} (or {\em window}) process corresponding to  $(f(t))_{t\ge-r_0}$.
For each $p\ge 1,$ let $\mathscr P_p(\C)$ be the set of all probability measures on $\C$, denoted by $\mathscr P(\C)$,  with finite $p$-th moment, i.e.,
 $$\mathscr P_p(\C) =\bigg\{\mu\in\mathscr P(\C)\bigg|\,  \mu\big(\|\cdot\|_\8^p\big):=\int_\C\big\|\xi\big\|_\8^p\mu(\d\xi)   <\8\bigg\},$$
 which is a Polish space (see e.g. \cite[Theorem 6.18]{Villani}) under the Wasserstein distance
 $$\mathbb W_p(\mu,\nu)=\inf_{\pi\in\mathcal C(\mu,\nu)}\bigg(\int_{\C\times\C}\|\xi-\eta\|_\infty^p\pi(\d\xi,\d\eta)\bigg)^{\ff{1}{p}}, $$
where $\mathcal C(\mu,\nu)$ stands for the set of all probability measures on $\C\times\C$ with marginals $\mu$ and $\nu$, respectively.
 For a random variable $\xi$, we denote its law by $\L_{\xi}$ and we write $\xi\sim \mu$ if $\mu=\L_{\xi}$.

 We fix $p\ge 2$ and consider the following functional   McKean-Vlasov  SDE
 \begin{equation}\label{E1}
 \d X(t)=b(X_t,\L_{X_t})\d t+\si(X_t,\L_{X_t})\d W(t),~~~t\ge0,~~~X_0\sim\mu\in\mathscr P_p(\C),
 \end{equation}
 where $$b:\C\times\mathscr P_p(\C)\to\R^d, \quad \si:\C\times\mathscr P_p(\C)\to\R^d\otimes\R^m,$$ and $(W(t))_{t\ge0}$ is an $m$-dimensional Brownian motion on the  probability space $(\OO,\F, \P)$ with the filtration $(\F_t)_{t\ge0}$.
We will assume the following hypotheses.
\begin{enumerate}
\item[$({\bf H_1})$] $b$ and $\si$ are continuous and bounded    on bounded subsets
of $\C\times\mathscr P_p(\C)$;
\item[$({\bf H_2})$]  There is  a constant   $K> 0$ such that \begin{align*}
2\<\xi(0)-\eta(0),b(\xi,\mu)-b(\eta,\nu)\>^++\|\si(\xi,\mu)-\si(\eta,\nu)\|_{\rm HS}^2
\le K\big\{\|\xi-\eta\|_\8^2 +  \mathbb W_p(\mu,\nu)^2\big\}
\end{align*}
for any $ \xi,\eta\in\C, \mu,\nu\in\mathscr P_p(\C).$
\end{enumerate}

We will be interested in  the existence
of an IPM $\pi \in \mathscr P_p(\C)$ of \eqref{E1}, i.e.~a probability measure $\pi$ for which the functional solution $(X_t)_{t\ge0}$ of
$$
\d X(t)=b(X_t,\pi)\d t+\si(X_t,\pi)\d W(t),~~~t\ge0,~~~X_0\sim\pi,
$$
satisfies $\L_{X_t}=\pi$ for every $t\ge0$. In this case we say that \eqref{E1} (or $(X_t)_{t\ge0}$) {\em admits an IPM}.

 To investigate existence
of an IPM we do not need to know whether  \eqref{E1} has a unique (strong) solution for every initial condition but we still mention that it has been shown in  \cite[Lemma 3.1]{BRW} that
under $({\bf H_1})$ and $({\bf H_2})$,    for any $X_0\in L^p(\OO\to\C,\F_0,\P)$ with $p\ge2$, \eqref{E1} has a unique functional solution $(X_t)_{t\ge0}$ and that there exists a nondecreasing positive function  $T\mapsto C_T $ such that
  \begin{equation*}
  \E\Big(\sup_{0\le t\le T}\|X_t\|_\8^p\Big)\le C_T\E\|X_0\|_\8^p,~~~T>0.
  \end{equation*}

 Clearly, hypotheses $({\bf H_1})$ and $({\bf H_2})$ are insufficient to guarantee  existence of an IPM for \eqref{E1}.
Therefore, we impose the following additional condition which
  guarantees that the drift $b$ pushes solutions towards {\bf0} whenever $\|X_t\|_\infty$ is large. %and that $\sigma$ does not increase too quickly.
 \begin{enumerate}
\item[$({\bf H_3})$] There exist a constant  $\ll_1>0$ and  constants $\ll_0,\ll_i\ge0, i=2,\cdots,5,$ such that
\begin{align}
2\<\xi(0),b(\xi,\mu)\>&\le\ll_0-\ll_1|\xi(0)|^2+\ll_2\|\xi\|_\8^2+\ll_3\W_p(\mu,\dd_{\bf0})^2,\label{W4}\\
\|\si(\xi,\mu)\|_{\rm HS}^2&\le\ll_0+ \ll_4\|\xi\|_\8^2+\ll_5\W_p(\mu,\dd_{\bf0})^2\label{W5}
\end{align}
 for any $\xi\in\C$ and $\mu\in\mathscr P_p(\C)$.
\end{enumerate}
Note that \eqref{W5} actually follows from $({\bf H_2})$ (for suitable $\lambda_0$, $\lambda_4$ and $\lambda_5$).

In the sequel, we will frequently use the following one-sided version of the Burkholder-Davis-Gundy inequality which is a special case of \cite[Theorem 4.1.(ii)]{ose}.
\begin{prp}\label{ose}
  For any continuous martingale $M$ satisfying $M(0)=0$ and for any $t \ge 0$, we have
  $$
\E\Big( \sup_{0 \le s \le t} M(s)\Big)\le \chi\, \E\big( [M,M]_t^{1/2} \big),
$$
where $\chi$ $($called $\nu_1$ in \cite[Theorem 4.1.$($ii$)$]{ose}$\,)$ is the smallest positive root of the confluent hypergeometric function with parameter 1.
\end{prp}
The numerical value is $\chi \approx 1.30693...$.\\

For the parameters $\ll_1>0$, and $\ll_2,\ll_3,\ll_4,\ll_5\ge 0$ in (${\bf H_3}$), set
\begin{equation*}
\begin{split}
\mathcal A:&=\big\{(\vv,\aa,\gamma_1,\gg_2)\in(0,1)\times(0,\8)^3\big|\aa-2\,\e^{\aa r_0}\big(\kk_1(\vv,\gg_1)+ \kk_2(\vv,\gg_2) \big)>0,\\ &\quad\quad\qquad\qquad\qquad\quad\qquad\qquad\qquad~\psi(\vv,\aa,\gg_1,\gg_2)<0\big\},
\end{split}
\end{equation*}
where
\begin{equation}\label{f-1}
\begin{split}
\psi(\vv,\aa,\gg_1,\gg_2):&=\frac{1-\vv}{\vv}\Big\{\alpha-  \ff{p\ll_1}{2} + \ff{p\gamma_1}{2} \big( \lambda_2 + \lambda_4  (p%p+ p\CC_1^2
-1)      \big)  +\ff{p\gamma_2 }{2}\big( \lambda_3 + \lambda_5  (p%p+p\CC_1^2
-1) \big)\Big\}\\
&\quad+\frac{\chi^2p^2}{2\vv^2}\big(\lambda_4\gg_1+\lambda_5\gg_2\big),\\
 \kk_1(\vv,\gg_1):&=\frac{1}{\vv} \big(\lambda_2 +\lambda_4  ((1+\chi^2/\vv)p%p+p\CC_1^2
 -1) \big)\Big(\ff{p-2}{p\gamma_1}\Big)^{\ff{p-2}{2}},\\
 \kappa_2(\vv,\gg_2):& =\frac{1}{\vv} \big(\lambda_3 +\lambda_5  ((
 1+\chi^2/\vv)p%p +p\CC_1^2
 -1)\big)\Big(\ff{p-2}{p\gamma_2}\Big)^{\ff{p-2}{2}}.
\end{split}
\end{equation}
Here and below, we interpret expressions of the form $0^0$ as 1 (when $p=2$).

Our   main result is the following theorem.
\begin{thm}\label{main}
Fix $p\ge 2$ and assume $({\bf H_1})$, $({\bf H_2})$, and $({\bf H_3})$. If $\mathcal A\neq\emptyset$,  then \eqref{E1} admits an IPM $\pi \in \mathscr P_p(\C)$.
\end{thm}

%\begin{rem}
%\red{We remark that in Theorem \ref{main} we merely address existence of an IPM for functional McKean-Vlasov SDEs, where  uniqueness of IPMs is left. Whereas, if some additional   condition, e.g., monotone condition, is imposed,  then, by a standard argument, we can also derive uniqueness of IMP. Concerning  uniqueness of IPMs of McKean-Vlasov SDEs, we refer to \cite{AD,RW,Wangb} under monotone conditions; see e.g. \cite[(H2')]{Wangb} for more details. Moreover, for non-uniqueness of IPMs, we would like to mention the Ornstein-Uhlenbeck type McKean-Vlasov SDE   \cite[(67)]{AD}.
%}
%\end{rem}

\begin{rem}
  Note that, for fixed  $\ll_2,\ll_3,\ll_4,\ll_5\ge 0$ and $p>2$, the set $\mathcal A$ is non-empty  if $\lambda_1$ is sufficiently large: to see this, fix $(\varepsilon, \alpha)\in (0,1)\times (0,\infty)$.
  Then choose $\gamma_1,\gamma_2$ so large that $\aa-2\,\e^{\aa r_0}\big(\kk_1(\vv,\gg_1)+ \kk_2(\vv,\gg_2) \big)>0$. Finally choose $\lambda_1>0$ so large that $\psi(\vv,\aa,\gg_1,\gg_2)<0$.

  For $p=2$ this argument fails (in general) but it is still true that for fixed  $\ll_2,\ll_3,\ll_4,\ll_5\ge 0$ and $p=2$ and $\lambda_1$ sufficiently large  \eqref{E1} admits an IPM $\pi \in \mathscr P_2(\C)$. To see this, replace $p=2$ by some $q>p$ and observe that conditions $({\bf H_1})$, $({\bf H_2})$, and $({\bf H_3})$ still hold true for $q$
  instead of $p$ with the same constants, provided we restrict the domains of $b$ and $\sigma$ accordingly. By arguing as above, we see that the set $\mathcal A$ defined in terms of $q$ instead of $p$
  is non-empty for $\lambda_1$ sufficiently large, so Theorem \ref{main} even guarantees the existence of an IPM $\pi \in \mathscr P_q(\C)\big(\subset  \mathscr P_2(\C)\big)$.
\end{rem}

\begin{rem}
  Let us now see how $\lambda_1$ has to grow as a function of $p$ for an IPM to exist. Fix $(\vv,\aa,\gamma_1,\gg_2)\in(0,1)\times(0,\8) \times (1,\infty)^2$ and fix
  $\ll_2,\ll_3,\ll_4,\ll_5\ge 0$.  Then $\kappa_1$ and $\kappa_2$ converge to 0 as $p \to \infty$ and it is  clear that there exists some $\beta>0$ such that for $\lambda_1>\beta p$ we have
  $\psi(\vv,\aa,\gg_1,\gg_2)<0$ for all $p \ge 2$. The assumption that $\lambda_1$ has to grow at least linearly in $p$ to obtain an IPM in $\mathscr P_p(\C)$ cannot be avoided in general
  (not even in the case of a one-dimensional SDE without delay and without dependence on the law of the process) as the example in Remark \ref{F2} shows.
\end{rem}

\begin{rem}
  Theorem \ref{main} provides a stronger result than the main result in \cite{ESV} even in the case when neither $b$ nor $\sigma$ depend on the law $\L_{X_t}$ (i.e.,~$\lambda_3=\lambda_5=0$) since the main result in
  \cite{ESV} requires that  the drift is superlinear while we allow that the drift has linear (negative)  growth of sufficient strength.
\end{rem}

\begin{rem}\label{F2}
  Let $d=m=1$, $\sigma>0$, $\lambda_1>0$ and consider the SDE
  $$
\d X(t)=-\lambda_1 X(t)\,\d t +\sigma \sqrt{|X(t)|^2+1}\,\d W(t)
$$
(without delay and without dependence on the law of the process).
The corresponding real-valued Markov process admits an invariant measure, the so-called {\em speed measure}, with density
$$
f(x)=\frac 1{x^2+1} \exp\Big\{- 2\int_0^{|x|} \frac {\lambda_1u}{\sigma^2(u^2+1)}\,\d u\Big\}=\big(x^2+1\big)^{-\frac {\lambda_1}{\sigma^2}-1},\quad x \in \R
$$
(see \cite[p.343 \& p.353]{KS}). This measure is finite and hence, after normalizing, an IPM. It has a finite $p$-th moment
if and only if $p-2\frac{\lambda_1}{\sigma^2}<1$ or, equivalently, $\lambda_1>\frac 12(p-1)\sigma^2$ showing that, in general, $\lambda_1$ has to increase at least linearly as a function of $p$ to obtain an
IPM in $\mathscr P_p(\C)$. Of course, in this example, an IPM exists for every $p \ge 2$ and every $\lambda_1>0$ but it does not have a finite $p$-th moment when $\lambda_1$ is too small.
\end{rem}

The remainder of this paper is organized as follows. In Section \ref{sec2}, we prepare several auxiliary lemmas which are crucial for the proof Theorem \ref{main}.   Section \ref{sec3} is devoted  to completing the proof of Theorem \ref{main}.

\section{Preliminary Lemmas}\label{sec2}
Under $({\bf H_1})$ and $({\bf H_2})$,
for fixed $\mu,\nu\in\mathscr P_p(\C)$, \cite[Theorem 2.3]{RS} shows that the following frozen   SDE with memory
\begin{equation}\label{E2}
\d X^{\mu,\nu}(t)=b(X^{\mu,\nu}_t,\nu)\d t+\si(X^{\mu,\nu}_t,\nu)\d W(t), ~~t\ge0,~~ X_0^{\mu,\nu}\sim \mu
\end{equation}
has a unique functional  solution $(X^{\mu,\nu}_t)_{t\ge 0}$ provided that $X_0^{\mu,\nu}$ is $\F_0$-measurable. Further, $(X^{.,\nu}_t)_{t\ge 0}$ is a Markov process which is even
Feller, i.e., the corresponding Markov semigroup $(P_t^\nu)_{t \ge 0}$ on $\C$ maps $C_b(\C)$, the set of  bounded continuous functions,  to $C_b(\C)$; see e.g.
\cite[Proposition 3.1]{ESV}.

\begin{lem}\label{Lem}
Assume $({\bf H_1})$ and $({\bf H_2})$. Then, for any  $\mu,\nu_1,\nu_2\in\mathscr P_p(\C)$,
\begin{equation}\label{W2}
\E\|X^{\mu,\nu_1}_t-X^{\mu,\nu_2}_t\|^p_\8\le C\W_p(\nu_1,\nu_2)^p t\e^{Ct},~~~t\ge0,
\end{equation}
where $X^{\mu,\nu_1}_0=X^{\mu,\nu_2}_0\sim\mu$ and  $C:=2(\chi^2+1)Kp^2$.
\end{lem}

\begin{proof}
  Let $\Psi(t)=X^{\mu,\nu_1}(t)-X^{\mu,\nu_2}(t)$, $t\ge-r_0$, and
  $$
M(t)=p\int_0^t\big|\Psi(s)\big|^{p-2}\big\<\Psi(s),\big(\si(X_s^{\mu,\nu_1},\nu_1)-\si(X_s^{\mu,\nu_2},\nu_2)\big)\,\d W(s)\big\>,~~~t\ge0.
  $$
For $N \in \N$, define $\tau_N=\inf\{s \ge 0:\,|\Psi(s)| \ge N\}$,
  $\Psi^N(t)=\Psi(t \wedge \tau_N)$,  and $M_N(t)=M(t \wedge \tau_N)$.
Note that $M_N$ is a martingale.
By It\^o's formula, it follows from $({\bf H_2})$ that
\begin{equation*}
\begin{split}
\d \big|\Psi(t)\big|^p&\le \ff{p}{2}\big|\Psi(t)\big|^{p-2}\big\{2\big\<\Psi(t),b(X_t^{\mu,\nu_1},\nu_1)-b(X_t^{\mu,\nu_2},\nu_2)\big\>\\
&\quad+  (p-1) \big\|\si(X_t^{\mu,\nu_1},\nu_1)-\si(X_t^{\mu,\nu_2},\nu_2)\big\|_{\rm HS}^2\big\}\,\d t+\d M(t)\\
&\le \ff{p^2K}{2}\big|\Psi(t)\big|^{p-2}   \big\{\|\Psi_t\|_\8^2+\W_p(\nu_1,\nu_2)^2\big\} \,\d t +\d M(t)\\
&\le pK\big\{(p-1)\|\Psi_t\|_\8^p+\W_p(\nu_1,\nu_2)^p\big\}\,\d t+\d M(t),
\end{split}
\end{equation*}
where we used Young's inequality in the last step.
Therefore, due to  $X^{\mu,\nu_1}_0=X^{\mu,\nu_2}_0$,
we have
\begin{equation}\label{W1}
\E \big[\|\Psi_t^N\|^p_\8\big]\le  pK\int_0^t\big\{(p-1)\E\big[\|\Psi_s^N\|_\8^p\big]+\W_p(\nu_1,\nu_2)^p\big\}\,\d s+\E\Big[\sup_{(t-r_0)^+\le s\le t} M_N(s)\Big].
\end{equation}
Next, Proposition \ref{ose},  $({\bf H_2})$, and Young's inequality yield
\begin{equation*}
\begin{split}
  \E\Big[&\sup_{(t-r_0)^+\le s\le t} M_N(s)\Big]=\E\Big[\sup_{(t-r_0)^+\le s\le t} \Big(M_N(s)-M_N\big( (t-r_0)^+\big)\Big)\Big]\\
  &\le \chi  p\,\E\bigg[\bigg(\int_{(t-r_0)^+}^t\1_{[0,\tau_N)}(s)    \big|\Psi^N(s)\big|^{2(p-1)}\big\|\si(X_s^{\mu,\nu_1},\nu_1)-\si(X_s^{\mu,\nu_2},\nu_2)\big\|_{\rm HS}^2\,\d s\bigg)^{1/2}\bigg]\\
&\le \chi p\sqrt{K}\, \E\bigg[\bigg(\|\Psi_t^N\|^p_\8\int_{(t-r_0)^+}^t \1_{[0,\tau_N)}(s)    \big|\Psi^N(s)\big|^{p-2}\big\{\|\Psi_s^N\|_\8^2+ \W_p(\nu_1,\nu_2)^2\big\}        \,\d s\bigg)^{1/2}\bigg]\\
&\le \chi p \sqrt{2K}  \E\bigg[\bigg(\|\Psi_t^N\|^p_\8\int_0^t\big\{\|\Psi_s^N\|_\8^p+ \W_p(\nu_1,\nu_2)^p\big\}\,\d s\bigg)^{1/2}\bigg]\\
&\le \ff{1}{2}\E\big[\|\Psi_t^N\|^p_\8\big]+\chi^2Kp^2\int_0^t\big\{\E\big[\|\Psi_s^N\|_\8^p\big]+ \W_p(\nu_1,\nu_2)^p\big\}\,\d s.
\end{split}
\end{equation*}
Inserting the estimate above back into \eqref{W1} implies
\begin{equation*}
\E\big[ \|\Psi_t^N\|^p_\8 \big]\le 2(\chi^2+1) K p^2\int_0^t\big\{\E\big[\|\Psi_s^N\|_\8^p\big]+ \W_p(\nu_1,\nu_2)^p\big\}\,\d s.
\end{equation*}
Applying Gronwall's inequality and letting $N \to \infty$ implies \eqref{W2}.
\end{proof}

Before we move on, let us  introduce some additional notation. Let
\begin{equation}\label{f-2}
\mathcal U=\Big\{(\vv,\aa,\gg_1,\gg_2)\in(0,1)\times(0,\8)^3\Big|\,2\kk_1(\vv,\gg_1)\e^{\aa r_0}<\aa, \quad \psi(\vv,\aa,\gg_1,\gg_2)<0 \Big\},
\end{equation}
where the functions $\kk_1  $ and $\psi$ were introduced in \eqref{f-1}. Note that $\mathcal A \subseteq \mathcal U$.
If $\mathcal U\neq\emptyset$, then  there exist  $(\vv,\aa,\gamma_1,\gg_2,\gg_3)\in(0,1)\times(0,\8)^4$ such that
\begin{equation}\label{s-5}
\begin{split}
\phi(\vv,\aa,\gg_1,\gg_2,\gg_3):&= \psi(\vv,\aa,\gg_1,\gg_2)+\frac{1}{2}\Big( \frac{1-\vv}{\vv}+\frac{\chi^2}{\vv^2}\Big) p^2\gamma_3 \lambda_0
<0,
\end{split}
\end{equation}
and
\begin{equation}\label{ss-2}
\aa>2\e^{\alpha r_0}\kk_1(\vv,\gg_1).
\end{equation}
Furthermore, we  set
\begin{equation}\label{ss-1}
\kk_3(\vv,\gg_3):= \frac{p\lambda_0}{\vv} \big(1+\chi^2/\vv\big) \Big(\ff{p-2}{p\gamma_3}\Big)^{\ff{p-2}{2}}. % \big(  1  +\CC_1^2   \big)
\end{equation}

\begin{lem}\label{importantlemma}
Assume $({\bf H_1})$, $({\bf H_2})$, and $({\bf H_3})$.
If $\mathcal U\neq\emptyset$, then
for any   $\mu,\nu \in \mathscr P_p(\C)$,
\begin{equation}\label{ss-3}
\begin{split}
\E\|X_t^{\mu,\nu}\|_\infty^p &\le \ff{2  \e^{ \alpha r_0}}{\aa-2  \kappa_1 (\vv,\gg_1) \e^{\alpha r_0}}\big(\kk_3(\vv,\gg_3)+\kappa_2(\vv,\gg_2)  \W_p(\nu,\dd_{\bf0})^p \big) \\
 &\quad+\e^{\alpha r_0}\big(1+4/\vv+4    \kappa_1 (\vv,\gg_1) r_0 \e^{\alpha r_0}\big)\mu(\|\cdot\|^p_\8)\e^{-(\aa-2  \kappa_1 (\vv,\gg_1) \e^{\alpha r_0})t},\quad t\ge0,
\end{split}
\end{equation}
 where $\kappa_1 (\cdot)$ and $\kappa_2(\cdot)$ were introduced in \eqref{f-1}.
\end{lem}

\begin{proof}
Since $\mathcal U\neq\emptyset$, there exist $(\vv,\aa,\gg_1,\gg_2,\gg_3)\in(0,1)\times(0,\8)^4$ such that \eqref{s-5} and \eqref{ss-2} hold.   In the sequel, we shall stipulate  the parameters
 $(\vv,\aa,\gg_1,\gg_2,\gg_3)\in(0,1)\times(0,\8)^4$ satisfying \eqref{s-5} and \eqref{ss-2}  and fix $\mu,\nu \in  \mathscr P_p(\C)$ for $p\ge2$. Below, for notational simplicity,  we shall write $X(t)$ and $X_t$ instead of $X^{\mu,\nu}(t)$ and $X_t^{\mu,\nu}$, respectively.
By  It\^o's formula, it follows from $({\bf H_3})$ that
\begin{equation}\label{s-2}
\begin{split}
\d \big(\e^{\aa t}|X(t)|^p\big)&\le\e^{\aa t}\big\{\aa |X(t)|^p+p/2|X(t)|^{p-2}\big(2\<X(t),b(X_t,\nu)\>\\
 &\quad + (p-1)\|\si(X_t,\nu)\|_{\rm HS}^2\big)\big\}\,\d t +\d N(t)\\
                                         &\le \e^{\aa t}\Big\{ \big(\alpha-  p\lambda_1/2 \big)\big|X(t)\big|^p+\frac {p^2\lambda_0}2 \big|X(t)\big|^{p-2}\\
  &\quad +  \ff{p}{2}\big(\lambda_2 +\lambda_4  (p-1)   \big)
 \|X_t\|_\infty^2   \big|X(t)\big|^{p-2}     \\
 &\quad +  \ff{p}{2}\big(\lambda_3 +\lambda_5  (p-1)   \big) \W_p(\nu,\dd_{\bf0})^2  \big|X(t)\big|^{p-2}\Big\}\,\d t +\d N(t),
\end{split}
\end{equation}
where
$$N(t):=p\int_0^t\e^{\aa s}|X(s)|^{p-2}\<X(s),\si(X_s,\nu)\,\d W(s)\>,\quad t\ge0.$$
Applying Young's inequality in case $p>2$, we obtain
\begin{equation}\label{a-0-0}
\begin{split}
 \|X_t\|_\infty^2   \big|X(t)\big|^{p-2}&\le  \gamma_1  \big|X(t)\big|^p+\ff{2}{p} \Big(\ff{p-2}{p\gamma_1}\Big)^{\ff{p-2}{2}}\|X_t\|_\infty^p,\\
\W_p(\nu,\dd_{\bf0})^2  \big|X(t)\big|^{p-2}&\le \gamma_2  \big|X(t)\big|^p+\ff{2}{p} \Big(\ff{p-2}{p\gamma_2}\Big)^{\ff{p-2}{2}} \W_p(\nu,\dd_{\bf0})^p,\\
\big|X(t)\big|^{p-2}&\le  \gamma_3  \big|X(t)\big|^p+\ff{2}{p} \Big(\ff{p-2}{p\gamma_3}\Big)^{\ff{p-2}{2}}.
\end{split}
\end{equation}
Note that these inequalities also hold in case $p=2$ (due to our convention that $0^0=1$).
Thus, we derive from \eqref{s-2} that
\begin{equation*}
\begin{split}
\d \big(\e^{\aa t}|X(t)|^p\big)&\le \e^{\aa t}\Big\{ \Theta(\alpha,\gg_1,\gg_2,\gg_3)\big|X(t)\big|^p  + \big(\lambda_2 +\lambda_4  (p-1)   \big)  \Big(\ff{p-2}{p\gamma_1}\Big)^{\ff{p-2}{2}}
 \|X_t\|_\infty^p     \\
 &\quad   +p\lambda_0\Big(\ff{p-2}{p\gamma_3}\Big)^{\ff{p-2}{2}}+   \big(\lambda_3 +\lambda_5  (p-1)   \big)   \Big(\ff{p-2}{p\gamma_2}\Big)^{\ff{p-2}{2}} \W_p(\nu,\dd_{\bf0})^p \Big\}\,\d t +\d N(t),
\end{split}
\end{equation*}
in which
$$\Theta(\alpha,\gg_1,\gg_2,\gg_3):=\alpha- \ff{ p\lambda_1}{2 }+\frac {p^2\lambda_0\gamma_3}2+\ff{p\gamma_1}{2}\big(\lambda_2 +\lambda_4  (p-1)   \big)+\ff{p\gamma_2}{2}\big(\lambda_3 +\lambda_5  (p-1)   \big).$$
We will use the following basic fact: if $f,\tilde f:[a,b] \to [0,\infty)$ satisfy
$$
\sup_{a\le r\le b}\big(f(r)+\tilde f(r)\big)\le A,
$$
$\tilde f$ is non-decreasing, and $\kappa \in (0,1)$, then
\begin{equation*}
\sup_{a\le r\le b}f(r)\le A/\kappa-\big(1/\kappa-1\big)\tilde f(b).
\end{equation*}
In the following computation, it is at first not clear that the expected values are finite. This can, however, be shown by a stopping argument like in the proof of Lemma \ref{Lem}.

Since $\Theta(\alpha,\gg_1,\gg_2,\gg_3)<0$ due to \eqref{s-5}, we have
\begin{equation}\label{s-44}
\begin{split}
g(t):&=\E\Big(\sup_{(t-r_0)^+\le r\le t}\big(\e^{\aa r}|X(r)|^p\big)\Big)\\
&\le\frac{1}{\vv}\E|X(0)|^p+\frac{1}{\vv}\big(\lambda_2 +\lambda_4  (p-1)   \big)  \Big(\ff{p-2}{p\gamma_1}\Big)^{\ff{p-2}{2}}\int_0^t
 \e^{\aa s}\E\|X_s\|_\infty^p\,\d s\\
&\quad+\ff{1}{ \vv}\bigg(p\lambda_0\Big(\ff{p-2}{p\gamma_3}\Big)^{\ff{p-2}{2}}    +   \big(\lambda_3 +\lambda_5  (p-1)   \big)   \Big(\ff{p-2}{p\gamma_2}\Big)^{\ff{p-2}{2}} \W_p(\nu,\dd_{\bf0})^p\bigg)\int_0^t\e^{\aa s}\,\d s\\
 &\quad+\ff{1-\vv}{\vv}\Theta(\alpha,\gg_1,\gg_2,\gg_3) \int_{(t-r_0)^+}^t\e^{\alpha s}\E|X(s)|^p\,\d s \\
 &\quad+\frac{1}{\vv}\E \Big(\sup_{(t-r_0)^+\le r\le t}N(r)\Big).
\end{split}
\end{equation}
Next,
applying Proposition \ref{ose} and Young's inequality, we obtain
\begin{equation*}\label{s-1}
\begin{split}
  &\E \Big(\sup_{(t-r_0)^+\le r\le t}N(r)\Big)\\
  &\le \chi  p\E\bigg(\sup_{(t-r_0)^+\le r\le t }\big(\e^{\aa r}|X(r)|^p\big)\int_{(t-r_0)^+}^t\e^{\aa s}|X(s)|^{p-2}\|\si(X_s,\nu)\|_{\rm HS}^2\,\d s\bigg)^{1/2}\\
&\le \chi p \bigg\{\ff{\vv}{2\chi p} g(t)+\frac{\chi p}{2\vv}  \E\int_{(t-r_0)^+}^t\e^{\aa s}|X(s)|^{p-2}\|\si(X_s,\nu)\|_{\rm HS}^2\,\d s\bigg\}\\
&\le \ff{\vv}{2}g(t)  +\frac{\chi^2p^2}{2\vv} \E\int_{(t-r_0)^+}^t\e^{\aa s}|X(s)|^{p-2} \big( \lambda_0   + \lambda_4    \|X_s\|_\infty^2 + \lambda_5 \W_p(\nu,\dd_{\bf0})^2  \big) \,\d s\\
 &\le \ff{\vv}{2}g(t)  +\frac{\chi^2 p^2}{2\vv} \int_{(t-r_0)^+}^t\e^{\aa s}\Big( \big(\lambda_0 \gamma_3   + \lambda_4\gg_1 + \lambda_5 \gamma_2\big) \E|X(s)|^p +\frac{2\lambda_4}{p}\Big(\ff{p-2}{p\gamma_1}\Big)^{\ff{p-2}{2}}
 \E\|X_s\|_\infty^p\\
  &\quad +\ff{2\lambda_5 }{p} \Big(\ff{p-2}{p\gamma_2}\Big)^{\ff{p-2}{2}} \W_p(\nu,\dd_{\bf0})^p+\ff{2\lambda_0}{p} \Big(\ff{p-2}{p\gamma_3}\Big)^{\ff{p-2}{2}}\Big) \,\d s,
\end{split}
\end{equation*}
where the last inequality is due to \eqref{a-0-0}.
Whence,  we deduce from \eqref{s-44} that
\begin{equation*}
\begin{split}
g(t)
  &\le \frac{1}{\vv}\E|X(0)|^p +\big(\kk_3(\vv,\gg_3)+\kappa_2(\vv,\gg_2) \W_p(\nu,\dd_{\bf0})^p \big) \int_0^t \e^{\aa s} \,\d s \\
   &\quad+ \phi(\vv,\aa,\gg_1,\gg_2,\gg_3) \int_{(t-r_0)^+}^t \e^{\aa s}\E |X(s)|^p\,\d s   \\
  &\qquad + \kappa_1(\vv,\gg_1)  \int_0^t \e^{\aa s}\E\|X_s\|_\infty^p \,\d s    +  \ff{1}{2} g(t),
\end{split}
\end{equation*}
 where  $ \phi(\vv,\aa_1,\gg_1,\gg_2,\gg_3)$ is  given in \eqref{s-5}.
Since $ \phi(\vv,\aa_1,\gg_1,\gg_2,\gg_3)<0$, we obtain
\begin{equation*}\label{s-3}
\begin{split}
 g(t)
&\le \frac{2}{\vv}\E|X(0)|^p  + \ff{2}{\aa}\big(\kk_3(\vv,\gg_3)+\kappa_2(\vv,\gg_2) \W_p(\nu,\dd_{\bf0})^p \big)  (\e^{\aa t}-1)\\
&\quad+2 \kappa_1(\vv,\gg_1) \int_0^t \e^{\aa s} \E\Big(\sup_{s-r_0\le u \le s}|X(u)|^p \Big)\,\d s .
\end{split}
\end{equation*}
This, together with
\begin{equation*}
\begin{split}
\int_0^t \e^{\aa s}  \E\Big(\sup_{s-r_0\le u \le s}|X(u)|^p \Big)\,\d s&\le \e^{\aa r_0}\int_0^t   \E\Big(\sup_{s-r_0\le u \le s}(\e^{\aa u}|X(u)|^p) \Big)\,\d s\\
&\le  r_0\e^{\aa r_0}    \E\|X_0\|^p_\8   + \e^{\aa r_0}\int_0^t  g(s)\,\d s,
\end{split}
\end{equation*}
yields
$$
g(t) \le \Phi(t)+ 2\kappa_1(\vv,\gg_1) \e^{\alpha r_0}  \int_0^t g(u)\,\d u,
$$
where
\begin{equation}\label{s-4}
\begin{split}
\Phi(t):&= \ff{2}{\aa}\big(\kk_3(\vv,\gg_3)+\kappa_2(\vv,\gg_2) \W_p(\nu,\dd_{\bf0})^p \big)  \e^{\aa t}\\
  &\hspace{1cm}+2\Big(\frac{1}{\vv}\E|X(0)|^p +  \kappa_1 (\vv,\gg_1)\e^{\alpha r_0} r_0 \E  \|X_0\|_\infty^p\Big)\\
  & =:A\e^{\alpha t}+B.
  \end{split}
\end{equation}
Applying Gronwall's inequality, we get
\begin{equation}\label{geee}
g(t) \le \Phi(t)+2 \kappa_1 (\vv,\gg_1)\e^{\alpha r_0}\int_0^t \Phi(s)\e^{2  \kappa_1 (\vv,\gg_1) \e^{\alpha r_0}(t-s)}\,\d s.
\end{equation}
Observe that
\begin{align*}
%\begin{split}
\E\|X_t\|_{\infty}^p&\le \e^{-\aa(t-r_0)}\E\Big(\sup_{t-r_0\le s\le t}(\e^{\aa s}|X(s)|^p)\Big)\\
&\le \e^{-\aa(t-r_0)}\E\Big(\sup_{ t-r_0  \le s\le (t-r_0)^+}(\e^{\aa s}|X(s)|^p)\vee \sup_{(t-r_0)^+\le s\le t}(\e^{\aa s}|X(s)|^p)\Big)\\
&\le \e^{-\aa(t-r_0)}\big(\E\|X_0\|^p_\8+ g(t)\big)
%\end{split}*
\end{align*}
This, together with  \eqref{geee} as well as $\aa>2  \kappa_1 (\vv,\gg_1)\e^{\alpha r_0}$ in view of \eqref{ss-2}, implies
\begin{align*}
\E \|X_t\|_{\infty}^p &\le \e^{-\aa(t-r_0)}\Big(\E\|X_0\|^p_\8+ A\e^{\alpha t}+B\\
&\quad\quad\quad\qquad\quad+2 \kappa_1 (\vv,\gg_1) \e^{\alpha r_0}\int_0^t \big(A\e^{\alpha s}+B\big)\e^{2  \kappa_1 (\vv,\gg_1) \e^{\alpha r_0}(t-s)}\,\d s\Big)\\
&\le  \ff{ A  \aa \e^{ \alpha r_0}}{\aa-2  \kappa_1 (\vv,\gg_1) \e^{\alpha r_0}} +\big(\E\|X_0\|^p_\8+B\big) \e^{-\aa(t-r_0)}+ B\e^{\alpha r_0}\e^{-(\aa-2  \kappa_1 (\vv,\gg_1) \e^{\alpha r_0})t}\\
  &\le \ff{ A  \aa \e^{ \alpha r_0}}{\aa-2  \kappa_1 (\vv,\gg_1) \e^{\alpha r_0}}+\e^{\alpha r_0}\big(\E\|X_0\|^p_\8+2B\big)\e^{-(\aa-2 \kappa_1(\vv,\gg_1) \e^{\alpha r_0})t}.
\end{align*}
 As a result, \eqref{ss-3} follows by inserting the expressions of $A$ and $B$, given in \eqref{s-4}.
\end{proof}

Our next goal is to find a suitable compact subset $\mathcal K$ of  $\mathscr P_p(\C)$ to which we can apply Kakutani's fixed point theorem to obtain an IPM
of the Markov process generated by \eqref{E1}. Since a subset of $\mathscr P_p(\C)$ which is relatively compact with respect to the weak (sometimes called {\em narrow}) topology on the space
of probability measures on
$\C$ is not necessarily relatively compact in $\mathscr P_p(\C)$ endowed with the $\W_p$-Wasserstein distance, we have to work with different values of $p$ in the following. So far, we fixed the value of $p \ge 2$ and we will continue
to regard $p$ as fixed. Observe however, that
if $({\bf H_1})$, $({\bf H_2})$, and $({\bf H_3})$ hold for  a given $p\ge 2$, then they also hold for $p$ replaced by $q>p$ (with the same constants) by restricting the domain of definition of $b$ and $\sigma$
accordingly. Now we assume that the set $\mathcal A$  (defined in terms of $p$) is non-empty and we fix $(\varepsilon,\,\alpha,\,\gg_1,\,\gg_2)\in \mathcal A$;
then the same quadruple is also in the set $\mathcal A$ defined with respect to $q$ instead of $p$ provided that $q-p>0$ is
sufficiently small. This holds true since the functions $\psi$, $\kappa_1$ and $\kappa_2$ depend continuously on $p$.  From now on we fix such a quadruple and  $q>p$.
Then, in particular, Lemma \ref{importantlemma} holds for $q$. To avoid confusion, we write $\tilde \kappa_1$ if $p$ is replaced by $q$ and similarly for other functions.

For
\begin{equation}\label{f-3}
M\ge M_0:= \ff{2\tilde \kk_3(\vv,\gg_3)\e^{\aa r_0}}{\aa-2\e^{\aa r_0}(\tilde \kk_1(\vv,\gg_1)+\tilde \kk_2(\vv,\gg_2))}\vee \E\big[\|X_{r_0}^{\delta_0,\delta_0}\|_{\infty}^q\big],
\end{equation}
let
$$
\M_M=\big\{\mu \in \mathscr P_q(\C):\,\mu\big(\|\cdot\|_\infty^q\big) \le M\big\},
$$
and
$$
\mathcal K_0:=\M_M\cap\big\{\L_{X^{\mu,\nu}_{r_0}},\mu,\nu\in\M_M\big\}.
$$
Note that $\M_M$ is closed in  $\mathscr P_q(\C)$ and hence also in  $\mathscr P_p(\C)$ and that     $\mathcal K_0 \neq \emptyset$ since $\L_{X^{\delta_0,\delta_0}_{r_0}}\in \mathcal K_0$.
For   $\mu,\nu\in\mathscr P_p(\C)$,  write $\mu_t^\nu =\L_{X^{\mu,\nu}_t}$ for notational brevity and set
\begin{equation*}
\LL_\nu:=\big\{\mu\in\mathscr P_p(\C): \mu_t^\nu=\mu,~~~t\ge0\big\},
\end{equation*}
which is the collection  of all IPMs of $(X_t^{.,\nu})_{t\ge0}$ solving \eqref{E2}.
We will see in the next lemma, that we automatically obtain $\LL_\nu \subseteq \mathscr P_q(\C)$ if  $\nu\in\M_M$.

\begin{lem}\label{lemmma}
Assume $({\bf H_1})$, $({\bf H_2})$, $({\bf H_3})$ and assume that $\mathcal A\neq\emptyset$. Choose $q>p$ and $M_0$ as above and fix $M \ge M_0$. Then  $\LL_\nu\subseteq \mathcal K_0$ for any
 $\nu\in\M_M$.
\end{lem}

\begin{proof}
  Fix $\nu\in\M_M$. We write $(X_t^{\xi,\nu})_{t\ge0}$ in lieu of $(X_t^{\mu,\nu})_{t\ge0}$ when the initial distribution is $\mu=\dd_\xi$ for $\xi\in\C.$ Due to our assumption on $q$, we can
  apply Lemma \ref{importantlemma} with $q$ and obtain, for each
$\rho_\nu\in\LL_\nu$, $N>0$, and a suitable function $\beta(t),\,t \ge 0$, decreasing to 0, that
\begin{equation*}
\begin{split}
\int_\C \big( \|\xi\|_\8^q&\wedge N\big)\,\rho_\nu(\d\xi)=\int_\C\E\big(\|X_t^{\xi,\nu}\|_\8^q\wedge N\big)\,\rho_\nu(\d\xi)\\
&\le \int_\C\big(N\wedge\E\|X_t^{\xi,\nu}\|_\8^q \big)\,\rho_\nu(\d\xi)\\
&\le \ff{2  \e^{ \alpha r_0}}{\aa-2  \tilde \kappa_1 (\vv,\gg_1) \e^{\alpha r_0}}\big(\tilde \kk_3(\vv,\gg_3)+\tilde \kappa_2(\vv,\gg_2)  \W_q(\nu,\dd_{\bf0})^q  \big)\\
&\quad+\int_\C N \wedge \Big(\beta (t)|\xi\|^q_\8\Big)\,\rho_\nu(\d\xi),
\end{split}
\end{equation*}
where  the identity is due to the invariance of $\rho_\nu$ and the first inequality  holds true by  Jensen's inequality since $x\mapsto N\wedge x$ is a concave function.
Then, taking $t\to\8$  and using  the
Lebesgue dominated convergence theorem, one has
\begin{equation*}
\int_\C \big( \|\xi\|_\8^q\wedge N\big)\,\rho_\nu(\d\xi)\le\ff{2  \e^{ \alpha r_0}}{\aa-2  \tilde \kappa_1 (\vv,\gg_1) \e^{\alpha r_0}}\big(\tilde \kk_3(\vv,\gg_3)+\tilde \kappa_2(\vv,\gg_2)  \W_q(\nu,\dd_{\bf0})^q \big).
\end{equation*}
Applying the   monotone convergence theorem, we get
\begin{equation*}
  \int_\C   \|\xi\|_\8^q \,\rho_\nu(\d\xi)
\le\ff{2  \e^{ \alpha r_0}}{\aa-2  \tilde \kappa_1 (\vv,\gg_1) \e^{\alpha r_0}}\big(\tilde \kk_3(\vv,\gg_3)+\tilde \kappa_2(\vv,\gg_2)  \W_q(\nu,\dd_{\bf0})^q \big)=:C(\nu).
\end{equation*}
According to \eqref{f-3}, we infer  $C(\nu)\le M$ so that $\rho_\nu\in\M_M$ for $\nu\in\M_M $ and therefore
 $\LL_\nu\subseteq\M_M$. On the other hand, according to the structure of $\LL_\nu$, one has
$$\LL_\nu\subseteq\big\{\L_{X_{r_0}^{\mu,\nu}},\mu,\nu\in\M_M\big\}.$$
Thus, we arrive at $\LL_\nu\subseteq \mathcal K_0$ and the desired assertion follows.
\end{proof}

\begin{lem}\label{compact}
Under the same assumptions as in the previous lemma the closure $\overline{\mathcal K_0}$ of the set ${\mathcal K_0}$ in $\mathscr P_p(\C)$ is a compact subset of $\mathscr P_p(\C)$.
\end{lem}

\begin{proof}
The set $\overline{\mathcal K_0}$ is a closed subset of $\mathscr P_p(\C)$, so it suffices to show that
$\big\{\L_{X_{r_0}^{\mu,\nu}},\mu,\nu\in\M_M\big\}$ is a relatively compact set in $\mathscr P_p(\C)$ which follows once we have verified  the following two  conditions
(see \cite[Proposition 7.1.5 and (5.1.20)]{AGS})
\begin{enumerate}
\item[(a)]
\begin{equation*}
\sup_{\mu,\nu\in\M_M}\E\|X_{r_0}^{\mu,\nu}\|_\8^q<\infty,
\end{equation*}
  \item[(b)]
\begin{equation*}
  \big\{\L_{X_{r_0}^{\mu,\nu}},\mu,\nu\in\M_M\big\} \mbox{ is a tight subset of } \mathscr{P}(\C).
\end{equation*}
\end{enumerate}

Condition (a) is an immediate consequence of Lemma \ref{importantlemma} with $p$ replaced by $q$. Then,  (b) follows once we have verified the following property
(see e.g. \cite[Theorem 7.3, p.82]{Bill})

\begin{equation}\label{modulus}
\lim_{\dd\to0}\sup_{\mu,\nu\in\M_M}\P\Big(\sup_{0\le u\le v\le r_0, v-u\le \dd}|X^{\mu,\nu}(v)- X^{\mu,\nu}(u)|\ge\gg\Big)=0,\;\mbox{ for every } \gamma>0.
\end{equation}

By Chebyshev's inequality, for $\mu,\nu\in\M_M$, we obtain from \eqref{ss-3} that there exists a constant $c_1>0$ (dependent on $M$ but independent of $R$, $\mu$, and $\nu$) such that
\begin{equation*}
M^{\mu,\nu}_R:=\P\big(\|X^{\mu,\nu}_{r_0}\|_\8\ge R\big)+\P\big(\|X^{\mu,\nu}_0\|_\8\ge R\big)\le \ff{1}{R^p}\big\{\E\|X^{\mu,\nu}_{r_0}\|_\8^p+\mu(\|\cdot\|_\8^p)\big\} \le \ff{c_1}{R^p}.
\end{equation*}
Whence, for  any  $\vv_0>0$, there exists an $R_0=R_0(\vv_0,c_1,p)$ sufficiently large such that
\begin{equation}\label{A4}
\sup_{\mu,\nu\in\M_M}M^{\mu,\nu}_{R_0}\le \vv_0.
\end{equation}
Fix $\vv_0>0$ and $R_0>0$ satisfying \eqref{A4} and define the stopping time
$$
\tau^{\mu,\nu}:=\inf\big\{t \ge -r_0:\,\big|X^{\mu,\nu}(t)\big|>R_0\big\}.
$$
Then $\P\big(\tau^{\mu,\nu}\le r_0\big)\le \vv_0$ and, for $\gamma>0$, we get
\begin{equation*}
\begin{split}
  \P&\Big(\sup_{0\le u\le v\le r_0, v-u\le \dd}|X^{\mu,\nu}(v)- X^{\mu,\nu}(u)|\ge\gg\Big)\\
  &\le \P\big(\tau^{\mu,\nu}\le r_0\big) + \P\Big(\sup_{0\le u\le v\le r_0, v-u\le \dd} \int_u^v\1_{(0,\tau^{\mu,\nu}]} (s) |b(X^{\mu,\nu}_s,\nu)|\d s \ge\ff{\gg}{2}\Big)\\
&\quad+ \P\bigg(\sup_{0\le u\le v\le r_0, v-u\le \dd}\Big|\int_u^v \1_{(0,\tau^{\mu,\nu}]} (s)   \si(X^{\mu,\nu}_s,\nu)\d W( s)\Big|\ge\ff{\gg}{2}\bigg)\\
&=: \P\big(\tau^{\mu,\nu}\le r_0\big) +N^{\mu,\nu}(\dd)    +J^{\mu,\nu}(\dd).
\end{split}
\end{equation*}

Since $b$ is bounded on bounded subsets of $\C\times\mathscr P_p(\C)$    by $({\bf H_1})$, we obtain
\begin{equation}\label{A7}
\lim_{\dd\to0}\sup_{\mu,\nu\in\M_M}N^{\mu,\nu}(\dd)=0.
\end{equation}
If we further have
\begin{equation}\label{A6}
\lim_{\dd\to0}\sup_{\mu,\nu\in\M_M}J^{\mu,\nu}(\dd)=0.
\end{equation}
then the claim (b) holds by combining \eqref{A4}, \eqref{A7}, and  \eqref{A6} and by invoking the fact that $\vv_0>0$ is arbitrary. It remains to show \eqref{A6}.
This is in fact a standard argument (see, e.g.~ \cite{ESV}) since the integrand of the stochastic integral is bounded uniformly in $s,\,\mu,\,\nu,\,\omega$, so each component of the stochastic integral is a time-changed
Brownian motion where the time change process has a derivative which is uniformly bounded away from 0. Therefore, \eqref{A6} follows and the proof of the lemma is complete.
\end{proof}

\section{Proof of Theorem \ref{main}}\label{sec3}
The key idea of the proof of Theorem \ref{main} is to apply the following Kakutani's fixed point theorem for multivalued maps (see e.g.~\cite[Lemma 20.1, p20]{OR}):
\begin{lem}\label{ThA}
Let $\mathcal K$ be  a non-empty convex compact subset of a locally convex topological vector space and assume that $ \GG: \mathcal K\to 2^{\mathcal K}$ is a set-valued function $($where
$2^{\mathcal K}$ is the set of subsets of $\mathcal K$$)$ satisfying the following conditions:
\begin{enumerate}
\item[$(a)$] $\GG(x)$ is nonempty and convex for every $x\in \mathcal K$;

\item[$(b)$] The graph $Gr(\GG):=\{(x, y)\in \mathcal K\times \mathcal K : y\in\GG(x)\}$ is a closed subset of $\mathcal K \times \mathcal K$.
\end{enumerate}
Then,   $\GG$ admits a fixed point $\xi^*\in \mathcal K$, i.e., $\xi^*\in \GG(\xi^*).$
\end{lem}

Let
 $$\mathcal M^p=\big\{\mu| \, \mu \mbox{ is a finite signed measure on }  \C \mbox{ such that } |\mu|(\|\cdot\|_\infty^p)<\8\big\},$$
which is a normed space so it is a Hausdorff locally convex space
under the  Kantorovich-Rubinstein metric
$$\mathcal W_p(\mu,\nu)=\Big(\sup_{(f,g)\in\mathscr F_{\mathrm Lip}}\big(\mu(f)-\nu(g)\big)\Big)^{1/p},\quad  \mu,\nu\in\mathcal M^p,$$
where
 \begin{equation*}
 \mathscr F_{\mathrm Lip}=\big\{(f,g): f,g \mbox{ are Lipschitz such that } f(\xi)\le g(\eta)+\|\xi-\eta\|_\infty^p,\xi,\eta\in\C\big\}.
 \end{equation*}
 Observe that
 $$\mathcal W_p(\mu,\nu)=\mathbb W_p(\mu,\nu),\quad \mu,\nu\in \mathscr P_p(\C),$$
see, for example, \cite{Ra}. Therefore, Lemma  \ref{ThA} is applicable to the present set-up.

With Lemma \ref{ThA} at hand, we are in the position to complete the proof of Theorem \ref{main}.

\begin{proof}[Proof of Theorem \ref{main}]
  Let $\mathcal K$ be the convex hull of $\bar{\mathcal K_0}$ in $\mathscr P_p(\C)$, so $\mathcal K$ is a convex subset of $\M_M$.
  Further, by Lemma \ref{compact},  $\bar{\mathcal K_0}$ is a compact subset of $\mathscr P_p(\C)$. Thus, \cite[Theorem 5.35]{AB} enables us to deduce that $\mathcal K$ is also compact
  in $\mathscr P_p(\C)$. So we conclude that $\mathcal K$ is a nonempty convex compact subset of $\mathscr P_p(\C)$.

By taking advantage of Lemma \ref{importantlemma}, together with $({\bf H_1})$ and  $({\bf H_2})$, it follows from \cite[Theorem 2.3]{ESV} (see also \cite[Theorem 3.1.1]{KW}) that  the (Feller) semigroup  $(P_t^{\nu})$ admits at least one IPM
  in the set $\M_M$, so the set $\Lambda_\nu$ of IPMs of  $(P_t^{\nu})$ in $\M_M$ is non-empty.
  Define the multivalued map
 $\GG: \mathcal K\to 2^{\mathcal K}$   by
\begin{equation}\label{E3}
\GG(\nu)=\LL_\nu, ~~~\nu\in\mathcal K.
\end{equation}
As long as  the map $\GG$ has a fixed point, i.e., there exists a probability measure $\mu\in\mathcal K$ such that  $\mu\in\GG(\mu)$,  the definition of $\LL_\mu$ enables us to get $$\L_{X^{\mu,\mu}_t}=\mu_t^\mu=\mu, ~~~~t\ge0,  $$
so the $\C$-valued Markov process  $(X_t)_{t\ge0}$ solving \eqref{E1} has an IPM. To complete  the proof of Theorem \ref{main}, it is sufficient to show that the
multivalued map $\Gamma$, defined in \eqref{E3}, has a fixed point.  To apply Kakutani's fixed point theorem, we need to verify
\begin{itemize}
\item[({\bf i})]$\Lambda_\nu$ is convex for each $\nu \in \mathcal K$;

\item[({\bf ii})] the graph of $\Gamma$ is closed.
\end{itemize}
The claim ({\bf i})   is clear since convex combinations of IPMs are invariant and $\mathcal K$ is convex. To show ({\bf ii}), consider  sequences $(\nu_n)_{n\ge 1}$ and $(\pi_n)_{n\ge 1}$ in $\mathcal K$ such that $\pi_n \in \Lambda_{\nu_n}$ and such that there exist $\nu$ and $\pi$ in $\mathcal K$
such that  $\lim_{n \to \infty}\W_p(\nu_n,\nu)=0$ and $\lim_{n \to \infty}\W_p(\pi_n,\pi)=0$. We have to show that $\pi \in \Lambda_{\nu}$.  To see this, for any  $t>0$, by $\pi_n \in \Lambda_{\nu_n}$, we have
 $$\int_\C f(\xi)\pi_n(\d\xi)=\int_\C P_t^{\nu_n} f(\xi)\pi_n(\d\xi), ~~~f\in {\rm Lip}_b(\C),$$
 where
 ${\rm Lip}_b(\C)$ means the set of bounded Lipschitz continuous functions on $\C$.
Thus,  one has
\begin{equation}\label{W17}
\int_\C f(\xi)\pi_n(\d\xi)=\int_\C P_t^{\nu} f(\xi)\pi_n(\d\xi)+\Upsilon_n,
\end{equation}
where
\begin{equation*}
\Upsilon_n:=\int_\C\big( P_t^{\nu_n} f(\xi)-P_t^{\nu} f(\xi)\big)\pi_n(\d\xi).
\end{equation*}
Taking advantage of  Lemma \ref{Lem},  we infer
\begin{equation*}
|\Upsilon_n|\le\|f\|_{\rm Lip} \int_\C\E\|X_t^{\dd_\xi,\nu_n}-X_t^{\dd_\xi,\nu}\|_\8\pi_n(\d\xi)\le \|f\|_{\rm Lip} \big(C t\e^{Ct}\big)^{1/p}\W_p(\nu_n,\nu).
\end{equation*}
This, together with $\W_p(\nu_n,\nu)\to0$, $\W_p(\pi_n,\pi)\to0$, and
the Feller property of $(P_t^\nu)_{t\ge0}$, allows us to conclude from \eqref{W17} that
$$\int_\C P_t^\nu f(\xi)\pi(\d\xi)=\int_\C f(\xi)\pi(\d\xi), ~~~f\in {\rm Lip}_b(\C).$$
Therefore, $\pi \in \mathscr P_p(\C)$ is an IPM  so that $\pi\in\Lambda_\nu$ and the proof of the theorem is complete.
\end{proof}

 \end{document}